\documentclass{amsart}
\usepackage{amsmath,amssymb,epsfig,amscd,amsthm}

\usepackage{verbatim}

\numberwithin{equation}{section}

\newtheorem{lemma}[equation]{Lemma}
\newtheorem{thm}[equation]{Theorem}
\newtheorem{conjecture}[equation]{Conjecture}
\newtheorem{cor}[equation]{Corollary}
\newtheorem{prop}[equation]{Proposition}
\newtheorem{fact}[equation]{Fact}
\newtheorem{example}[equation]{Example}

\newtheorem{defi}[equation]{Definition}

\theoremstyle{remark}

\newtheorem{remark}[equation]{Remark}

\newtheorem*{notation}{Notation}

\renewcommand{\bar}[1]{#1\llap{$\overline{\phantom{\rm#1}}$}}

\newcommand{\lra}{\longrightarrow}

\DeclareMathOperator{\cJ}{\mathcal{J}}

\DeclareMathOperator{\sep}{{sep}}
\DeclareMathOperator{\tor}{{tor}}
\DeclareMathOperator{\End}{{End}}

\DeclareMathOperator{\hhat}{{\widehat{h}}}

\DeclareMathOperator{\GL}{{GL}}

\DeclareMathOperator{\dirac}{\overline{\delta}}
\DeclareMathOperator{\goes}{\stackrel{\emph{w}}{\rightarrow}}

\newcommand{\N}{{\mathbb N}}
\newcommand{\T}{{\mathbb{T}}}

\newcommand{\Q}{{\mathbb Q}}
\newcommand{\R}{{\mathbb R}}
\newcommand{\C}{{\mathbb C}}
\newcommand{\E}{{\mathbb E}}

\newcommand{\Spec}{{Spec}}

\newcommand{\Fp}{{\mathbb{F}_p}}
\newcommand{\Fq}{{\mathbb{F}_q}}
\newcommand{\Fpbar}{{\bar{\Fp}}}

\newcommand{\Kbar}{{\bar{K}}}
\newcommand{\Qbar}{\bar{\mathbb{Q}}}

\newcommand{\bG}{{\mathbb G}}
\newcommand{\Gal}{{\rm Gal}}
\newcommand{\OO}{{\mathcal O}}

\newcommand{\isomto}{\overset{\sim}{\rightarrow}}
\newcommand{\tensor}{\otimes}

\renewcommand{\Kbar}{K^{\sep}}

\newcommand{\e}{{\rm E_v}}

\newcommand{\fo}{\mathfrak o}

\newcommand{\bP}{{\mathbb P}}

  {\end{list}}

\begin{document}



\title{Integral points for Drinfeld modules}

\author{Dragos Ghioca}
\address{
Dragos Ghioca\\
Department of Mathematics\\
University of British Columbia\\
Vancouver, BC V6T 1Z2\\
Canada
}
\email{dghioca@math.ubc.ca}


\begin{abstract}
We prove that in the backward orbit of a non-preperiodic point under the action of a Drinfeld module of generic characteristic there exist at most finitely many points $S$-integral with respect to another nonpreperiodic  point. This provides the answer (in positive characteristic) to a question raised by Sookdeo in \cite{Vijay}. We also prove that for each nontorsion point $z$ there exist at most finitely many torsion points which are $S$-integral with respect to $z$. This proves a question raised by Tucker and the author in \cite{newdrin}, and it gives the analogue of Ih's conjecture \cite{Ih} for Drinfeld modules.
\end{abstract}

\thanks{ Our research was partially supported by NSERC}


\maketitle


\section{Introduction}
\label{intro}

Let $k$ be a number field, let $S$ be a finite set of places of $k$ containing all its archimedean places, and let $\fo_{k,S}$ be the subring of $S$-integers contained in $k$. The following two conjectures were made by Ih (and refined by Silverman and Zhang) as an analogue of the classical diophantine problems of Mordell-Lang, Manin-Mumford, and Lang; for more details, see \cite{Ih}.

\begin{conjecture}
\label{Ih conjecture 1}
Let $A$ be an abelian variety defined over $k$, and let $A_S/\Spec(\fo_{k,S})$ be a model of $A$. Let $D$ be an
effective divisor on $A$, defined over $k$, at least one of whose irreducible components is not
the translate of an abelian subvariety by a torsion point, and let $D_S$ be its Zariski closure in
$A_S$. Then the set  of all torsion points of $A(\bar{k})$ whose closure in $A_S$
is disjoint from $D_S$, is not Zariski dense in $A$.
\end{conjecture}

The next conjecture is for algebraic dynamical systems, and it is modeled by Conjecture~\ref{Ih conjecture 1} for elliptic curves where torsion points are seen as preperiodic points for the multiplication-by-$2$-map. In general, for any rational map $f$, we say that $\alpha$ is a preperiodic point for $f$ if its orbit under $f$ is finite. As always in arithmetic dynamics, we denote by $f^n$ the $n$-th iterate of $f$. So, $\alpha$ is preperiodic if and only if there exist nonnegative integers $m\ne n$ such that $f^m(\alpha)=f^n(\alpha)$ (for more details on the theory of arithmetic dynamics we refer the reader to Silverman's book \cite{Silverman book}).

\begin{conjecture}
\label{Ih conjecture 2}
Let $f$ be a rational function of degree at least $2$ defined over $k$, and let $\alpha\in \bP^1(k)$ be non-preperiodic for $f$. Then there are only finitely many preperiodic points  which are $S$-integral with
respect to $\alpha$, i.e. whose Zariski closures in $\bP^1/\Spec(\fo_{k,S})$ do not meet the Zariski closure of $\alpha$.
\end{conjecture}

In \cite{Ih}, Baker, Ih and Rumely prove the first cases of the above conjectures. They prove Conjecture~\ref{Ih conjecture 1} for elliptic curves $A$, which in particular provides a proof of Conjecture~\ref{Ih conjecture 2} for Latt\`es maps $f$. Also in \cite{Ih}, the authors prove Conjecture~\ref{Ih conjecture 2} when $f$ is a powering map (same proof works when $f$ is a Chebyshev polynomial). The ingredients for proving their results are a strong form of equidistribution for torsion points of $1$-dimensional algebraic groups (such as elliptic curves or $\bG_m$), lower bounds for linear forms in (elliptic) logarithms, and a good understanding of the local heights associated to a dynamical system coming from an algebraic group endomorphism. At this moment it seems difficult to extend the above results beyond the case of $1$-dimensional abelian varieties, or the case of rational maps associated to endomorphisms of $1$-dimensional algebraic groups.

It is natural to consider the above conjectures in characteristic $p$, but one has to be careful in their reformulation. Indeed, if $A$ is defined over $\Fpbar$ then all its torsion points are also defined over $\Fpbar$ and thus one can find infinitely many torsion points which are $S$-integral with respect to a divisor $D$ of $A$. Similarly, if $f$ is a rational map defined over $\Fpbar$, all its preperiodic points are also in $\Fpbar$ and thus again one can find infinitely many points which are $S$-integral with respect to a non-preperiodic point $\alpha$. On the other hand, the Drinfeld modules have always proven  to be the right analogue in characteristic $p$ of abelian varieties. Therefore we propose to study in this paper analogues of the above conjectures for Drinfeld modules.  One could consider an analogue of Conjecture~\ref{Ih conjecture 1} for $T$-modules acting on $\bG_a^n$, but similar to the case over number fields, proving any result towards Conjecture~\ref{Ih conjecture 1} (or its analogues) for groups varieties of dimension larger than $1$ would be difficult.

We start by defining Drinfeld modules (for more details, see Section~\ref{note}).  
Let $p$ be a prime number, let $q$ be a power of $p$, and let $K$ be a finite extension of $\Fq(t)$. A Drinfeld module $\Phi$ (of generic characteristic) is a ring homomorphism from $\Fq[t]$ to $\End_K(\bG_a)$. We fix an algebraic closure $\bar{K}$ of $K$, and we let $\Kbar$ be the separable closure of $K$ inside $\bar{K}$. 

Since each Drinfeld module is isomorphic (over $\Kbar$) to a Drinfeld module for which $\Phi_t:=\Phi(t)$ is a monic polynomial, we assume  from now on that $\Phi$ is indeed in \emph{normal form} i.e., $\Phi_t$ is monic. Note that a Drinfeld module $\Psi$ is isomorphic to $\Phi$ (over $\Kbar$) if there exists  $\gamma\in\Kbar$ such that $\Psi_t(x)=\gamma^{-1}\Phi_t(\gamma x)$. So, our results about $S$-integral points  are not affected by replacing $\Phi$ with an isomorphic Drinfeld module since conjugating by $\gamma$ only affects the finitely many places where $\gamma$ is not an $S$-unit (for a precise definition for $S$-integrality we refer the reader to Subsection~\ref{S-integrality subsection}). 

The points of $\Kbar$ which have finite orbit under the action of $\Phi$ are called torsion;  we denote by $\Phi_{\tor}$ the set of all torsion points for $\Phi$.

A Drinfeld module may have complex multiplication (similar to the case of abelian varieties), i.e. there exist endomorphisms $g$ of $\bG_a$ defined over $\Kbar$ such that $g\circ \Phi_t =\Phi_t\circ g$. In our results we assume $\Phi$ does not have complex multiplication since we employ a strong equidistribution theorem from \cite{equi}  for torsion points of Drinfeld modules which uses the assumption that $\Phi$ has no complex multiplication.

The places of $K$ split in two categories: \emph{infinite places} and \emph{finite places} depending on whether they lie over  the place $v_\infty$ of $\Fq(t)$, for which $v_\infty(f/g)=\deg(g)-\deg(f)$, for all nonzero $f,g\in\Fq[t]$. We assume $\Phi$ has good reduction at all its finite places, i.e., for all finite places $v$ of $K$, the coefficients of $\Phi_t$ are $v$-integral (recall that we already assumed that $\Phi_t$ is monic). Also, for each place $v$ of $K$ we fix an extension of it to $\Kbar$. Then we can prove the following result.

\begin{thm}
\label{Ih's conjecture result}
Assume $\Phi$ is in normal form and it has good reduction at all finite places of $K$, and also assume that $\Phi$ has no complex multiplication. Let $\beta\in K$ be a nontorsion point for $\Phi$, and let $S$ be a finite set of places of $K$. Then there exist at most finitely many $\gamma\in\Phi_{\tor}$ such that $\gamma$ is $S$-integral with respect to $\beta$.
\end{thm}

The proof of Theorem~\ref{Ih's conjecture result} goes through an intermediate result which offers an alternative way of computing the canonical height $\hhat_\Phi(x)$ of any point $x\in K$ (for more details, see Section~\ref{note}).
\begin{thm}
\label{alternative way of computing the canonical height}
Let $\Phi$ be a Drinfeld module as in Theorem~\ref{Ih's conjecture result}.  
Let $\beta\in K$, and let $\{\gamma_n\}\subset\Phi_{\tor}$ be an infinite sequence. Then
$$\sum_{v\in\Omega_K}\lim_{n\to\infty}\frac{1}{[K(\gamma_n):K]}\sum_{\sigma\in\Gal(\Kbar/K)} \log|\beta-\gamma_n^{\sigma}|_v=\hhat_\Phi(\beta).$$
\end{thm} 
In the above result, and also later in the paper, a sum involving $\delta^\sigma$ over all $\sigma\in\Gal(\Kbar/K)$ is simply a sum over all the Galois conjugates of $\delta$.

Theorem~\ref{alternative way of computing the canonical height} answers a conjecture of Tucker and the author (\cite[Conjecture 5.2]{newdrin}).  It is worth pointing out that there is no \emph{Bogomolov-type} statement of Ih's Conjecture for Drinfeld modules, i.e., it is possible to find infinitely many points $x$ of small canonical height which are $S$-integral with respect to a given nontorsion point $\beta$ (see Example~\ref{no Bogomolov-Ih}). 

Theorem~\ref{Ih's conjecture result} may be interpreted as follows: for a nontorsion point $\beta$ there exist at most finitely many $\gamma$ such that $\Phi_Q(\gamma)=0$ for some nonzero polynomial $Q\in\Fq[t]$, and $\gamma$ is $S$-integral with respect to $\beta$. In other words, there exist at most finitely many points in the backward orbit of $0$ (under $\Phi$) which are $S$-integral with respect to $\beta$. In general, for each $\alpha\in K$ we let $\OO_\Phi^-(\alpha)$ be the \emph{backward orbit} of $\alpha$ under $\Phi$, i.e. the set of all $\beta\in\Kbar$ such that there exists some nonzero $Q\in \Fq[t]$ such that $\Phi_Q(\beta)=\alpha$. So, Theorem~\ref{Ih's conjecture result} yields that for each torsion point $\alpha$ of $\Phi$, there exist at most finitely many $\gamma\in \OO_\Phi^-(\alpha)$ which are $S$-integral with respect to the nontorsion point $\beta$. 
 It is natural to ask whether Theorem~\ref{Ih's conjecture result} holds when the point $\alpha$ is nontorsion. 

\begin{thm}
\label{backward orbit result}
Let $\Phi$ be a Drinfeld module without complex multiplication, and let $\alpha,\beta\in K$ be nontorsion points. If $\alpha\notin\Phi_{\tor}$, then there exist at most finitely many $\gamma\in \OO_\Phi^-(\alpha)$ such that $\gamma$ is $S$-integral with respect to $\beta$.
\end{thm}

Theorem~\ref{backward orbit result} is an analogue for Drinfeld modules of a result conjectured by Sookdeo in \cite{Vijay}. Motivated by a question of Silverman for $S$-integral points in the forward orbit of a rational map defined over $\Qbar$, Sookdeo \cite[Conjecture~1.2]{Vijay} asks whether for a rational map $f$ defined over $\Qbar$, and for a given non-preperiodic point $\alpha$ of $\Phi$, there exist at most finitely many points $\gamma$ in the backward orbit of $\alpha$ (i.e., $f^n(\gamma)=\alpha$ for some nonnegative integer $n$) such that $\gamma$ is $S$-integral with a given non-preperiodic point $\beta$. In \cite{Vijay}, Sookdeo proves that his conjecture holds when $f$ is either a powering map, or a Chebyshev polynomial. For an arbitrary rational map $f$, the conjecture appears to be difficult. Furthermore, Sookdeo presents a general strategy of attacking his conjecture by reducing it to proving that the number of Galois orbits for $f^n(z)-\beta$ is bounded above independently of $n$.

We observe that Theorems~\ref{backward orbit result} and \ref{Ih's conjecture result} combine to yield the following result.

\begin{thm}
\label{combination result}
Let $\Phi$ and $K$ be as in Theorem~\ref{Ih's conjecture result}, and let $\alpha,\beta\in K$ such that $\beta$ is not a torsion point for $\Phi$. Then there exist at most finitely many $\gamma\in\OO_\Phi^-(\alpha)$ such that $\gamma$ is $S$-integral with respect to $\beta$.
\end{thm}

It is essential to ask that  $\beta$ is not a torsion point for $\Phi$ in Theorem~\ref{combination result} since otherwise the result is false. Indeed, if  $\beta$ were equal to $0$, say, then one can  find infinitely many points $\gamma\in\OO_\Phi^-(\alpha)$ which are $S$-integral with respect to $0$, for any  set of places $S$ which contains the infinite places and also all the places where $\alpha$ is not a unit. Indeed, if $\Phi_Q(\gamma)=\alpha$ for some nonzero $Q\in\Fq[t]$ then for each $v\notin S$, we have that if $|\gamma|_v>1$ then $|\alpha|_v>1$, while if $|\gamma|_v<1$ then $|\alpha|_v<1$ (because $\Phi$ has good reduction at $v$). Hence, each $\gamma\in\OO_\Phi^-(\alpha)$ is $S$-integral with respect to $0$.

Instead of considering $S$-integral points in the backward orbit of a point $\alpha$ under a rational map $f\in\Qbar(z)$, one could consider $S$-integral points in its forward orbit. This question was settled by Silverman \cite{Sil-siegel} who showed (as an application of Siegel's classical theorem \cite{Siegel} on $S$-integral points on curves) that if there exist infinitely many $S$-integers in the forward orbit of $\alpha$, then $f\circ f$ is totally ramified at infinity, i.e., $f\circ f$ is a polynomial. In our setting for Drinfeld modules, studying $S$-integrality of points of the form $\Phi_Q(\alpha)$ (for arbitrary $Q\in \Fq[t]$) relative to a given point $\beta$ was done by Tucker and the author in \cite{equi} and \cite{siegel} (where a Siegel-type theorem for Drinfeld modules was proven). The result from \cite{siegel} will be used in the proof of Theorem~\ref{backward orbit result}. The results of our present paper complete the picture for Drinfeld modules by studying  the $S$-integrality of points in the backward orbit of a point.

Here is the plan of our paper. In Section~\ref{note} we introduce the notation and also state a crucial result (of H\"aberli \cite{Haberli} and Pink \cite{Pink-Kummer}) on Galois orbits of points in the backward orbit of a given point $\alpha$ (see Theorem~\ref{bounded number of orbits}). This last result allows us to prove Theorem~\ref{backward orbit result}. In Section~\ref{preliminary torsion} we derive some results regarding torsion of Drinfeld modules which have everywhere good reduction away from the places at infinity. Finally, we conclude our proof of Theorems~\ref{Ih's conjecture result} and~\ref{alternative way of computing the canonical height} in Section~\ref{equi}.

\medskip

\emph{Acknowledgments.} We thank Patrick Ingram and Thomas Tucker for several discussions and suggestions for improving this paper.


\section{Generalities}
\label{note}

\subsection{Drinfeld modules}

Let $p$ be a prime number, and let $q$ be a power of $p$. We let $K$ be a finite extension of $\Fq(t)$, let $\bar{K}$ be a fixed algebraic closure of $K$ and $K^{\sep}$ be the separable closure of $K$ inside $\Kbar$. 

Let $\Phi:\Fq[t]\lra \End_K(\bG_a)$ be a Drinfeld module of generic characteristic, i.e., $\Phi$ is a ring homomorphism with the property that
$$\Phi(t)(x):=\Phi_t(x)=tx+\sum_{i=1}^r a_i x^{q^i}$$
where $r\ge 1$ and each $a_i\in K$. We call $r$ the rank of the Drinfeld module $\Phi$.

We note that usually a Drinfeld module is defined on a ring $A$ of functions defined on a projective irreducible curve $C$ defined over $\Fq$, which are regular away from a given point $\eta$ on $C$. In our definition, $C=\bP^1$ and $\eta$ is the point at infinity. There is no loss of generality for using our definition since $\Fq[t]$ embeds into each such ring $A$ of functions.

At the expense of replacing $K$ with a finite extension, and replacing $\Phi$ with an isomorphic Drinfeld module $\Psi=\gamma^{-1}\circ \Phi\circ \gamma$, for a suitable $\gamma\in\Kbar$, we may assume $\Phi_t$ is monic. We say in this case that $\Phi$ is in normal form.

\subsection{Endomorphisms of $\Phi$} 

We say that $f\in K^{\sep}[x]$ is an endomorphism of $\Phi$ if $\Phi_t \circ f = f\circ \Phi_t$, or equivalently if $\Phi_a\circ f = f\circ \Phi_a$ for all $a\in\Fq[t]$.  The set of all endomorphisms of $\Phi$ is $\End_{K^{\sep}}(\Phi)$. Generically, $\Phi$ has no endomorphism other than $\Phi_a$ for $a\in\Fq[t]$; in this case $\End_{K^{\sep}}(\Phi)\isomto \Fq[t]$.

\subsection{Torsion points}

For each nonzero $Q\in \Fq[t]$, the set of all $x\in\Kbar$ such that $\Phi_Q(x):=\Phi(Q)(x)=0$ is defined to be $\Phi[Q]$; each such point $x$ is called a torsion point for $\Phi$. The set of all torsion points for $\Phi$ is denoted by
$$\Phi_{\tor}:=\bigcup_{Q\in\Fq[t]\setminus\{0\}} \Phi[Q].$$ 
One can show that for each nonzero $Q\in\Fq[t]$, we have $\Phi[Q]\isomto (\Fq[t]/(Q))^r$. Moreover, since each $\Phi_Q$ is a separable polynomial, we obtain that $\Phi_{\tor}\subset K^{\sep}$. Furthermore, since any polynomial $f\in\bar{K}[z]$ satisfying $f\circ \Phi_t=\Phi_t\circ f$ has the property that $f(\Phi_{\tor})\subset \Phi_{\tor}$ we conclude that $f\in\Kbar[z]$; i.e., all endomorphisms of $\Phi$ are indeed defined over $\Kbar$. For more details on Drinfeld modules, see \cite{Goss}.

\subsection{Places of $K$}

Let $\Omega_K$ be the set of all places of $K$. The places from $\Omega_K$ lie above the places of $\Fq(t)$. We normalize each corresponding absolute value $|\cdot |_v$ so that we have a well-defined product formula for each nonzero $x\in K$:
$$\prod_{v\in M_K}|x|_v=1.$$
Furthermore, we may assume $\log|t|_{\infty}=1$, where $\infty:=v_\infty$ is the place of $\Fq(t)$ corresponding to the degree of rational functions, i.e. (in exponential form), $v_\infty(f/g)=\deg(g)-\deg(f)$ for any nonzero $f,g\in\Fq[t]$. 
We let $S_\infty$ be the set of (infinite) places in $\Omega_K$ which lie above $v_\infty$. If $v\in \Omega_K\setminus S_\infty$, then we say that $v$ is a finite place for $\Phi$. For each $v\in\Omega_K$, we fix a completion $K_v$ of $K$ with respect to $v$, and also we fix an embedding of $\Kbar$ into $\bar{K_v}$. Finally, we let $\C_v$ be the completion with respect to $|\cdot |_v$ of a fixed algebraic closure of $K_v$.

\subsection{$S$-integrality}
\label{S-integrality subsection}

For a set of places $S \subset M_K$ and $\alpha\in K$, we say
that $\beta \in\Kbar$ is $S$-integral with respect to $\alpha$ if for
every place $v\notin S$, and for every morphism $\sigma\in\Gal(\Kbar/K)$ the following
are true:
\begin{itemize}
\item if $|\alpha|_v\le 1$, then $|\alpha-\beta^{\sigma}|_v\ge 1$.
\item if $|\alpha|_v >1$, then $|\beta^{\sigma}|_v\le 1$.
\end{itemize}
For more details about the definition of $S$-integrality, we refer the reader to \cite{Ih}.

We note that if $\Psi$ is an arbitrary Drinfeld module and $\Phi$ is a normalized Drinfeld module isomorphic to $\Psi$ through the conjugation by $\gamma\in\Kbar$ (i.e., $\Phi_t(x)=\gamma^{-1}\Psi_t(\gamma x)$), then all torsion points of $\Phi$ are of the form $\gamma^{-1}\cdot z$, where $z\in \Psi_{\tor}$. So, if we let $K$ be a finite extension of $\Fp(t)$ containing $\gamma$, and let $S$ be a finite set of places of $K$ containing all  places where $\gamma$ is not a unit, then $\alpha\in \Psi_{\tor}$ is $S$-integral with respect to $\beta\in K$ if and only if $\gamma^{-1}\alpha\in \Phi_{\tor}$ is $S$-integral with respect to $\gamma^{-1}\beta$. Therefore, our hypothesis from Theorems~\ref{Ih's conjecture result} and \ref{backward orbit result} (since a similar statement holds for the backward orbit of a nontorsion point) that $\Phi$ is in normal form is not essential. However, we prefer to work with Drinfeld modules in normal form since this makes easier to define the notion of good reduction for $\Phi$, and also simplifies some of the technical points in our proof.

\subsection{Good reduction for $\Phi$}

We say that the Drinfeld module $\Phi$ defined over $K$ has good reduction at a place $v$ of $K$, if each coefficient of $\Phi_t$ is integral at $v$, and moreover, the leading coefficient of $\Phi_t$ is a $v$-adic unit. It is immediate to see that if $\Phi$ has good reduction at $v$, then for each nonzero $Q\in \Fq[t]$ we have that $\Phi_Q$ has all its coefficients $v$-adic integers, while its leading coefficient is a $v$-adic unit. Clearly, a Drinfeld module cannot have good reduction at a place in $S_\infty$. Since we assumed that $\Phi$ is in normal form, the condition of having good reduction at $v$ reduces to the fact that $\Phi_t$ has all its coefficients $v$-adic integers.

\subsection{The action of the Galois group}

The Galois group $\Gal(K^{\sep}/K)$ acts on each $\Phi[Q]$. In particular, $\Gal(K^{\sep}/K)$ acts on the associated Tate module $\T_{Q}(\Phi)$ for $\Phi$ corresponding to irreducible, monic $Q\in \Fq[t]$; the Tate module  is isomorphic to $\Fq[t]_{(Q)}^r$ (where $\Fq[t]_{(Q)}$ is the $(Q)$-adic completion of $\Fq[t]$ at the prime ideal $(Q)$). We let $\T:=\prod_Q \T_Q$ (where the product is over all the monic irreducible polynomials $Q\in\Fq[t]$) be the profinite completion of $\Fq[t]$. Therefore we obtain a continuous representation
\begin{equation}
\label{this is rho}
\rho:\Gal(K^{\sep}/K)\lra \GL_r(\T).
\end{equation}
Pink and R\"{u}tsche \cite{Pink-Rutsche} proved that the image of $\rho$ is open, assuming that  $\Phi$ has no complex multiplication; their result is a Drinfeld module analogue of the classical Serre Openess Conjecture for (non-CM) abelian varieties. Initially, Breuer and Pink \cite{Breuer-Pink} proved that the image of $\rho$ is open assuming $K$ is a transcendental extension of $\Fq(t)$, which can be viewed as the function field version of the result from \cite{Pink-Rutsche}. In particular, Pink-R\"utsche result yields that for any torsion point $\gamma$ of order $Q$, where $\Q(t)\in\Fq[t]$ is a polynomial of degree $d$, there exists a positive constant $c_\Phi$ such that
$$\frac{[K(\gamma):K]}{\# \GL_r(\Fq[t]/(Q))}\ge c_\Phi.$$
Hence, $[K(\gamma):K]\gg q^{rd}$.

Now, let $x\in K$ be a nontorsion point for $\Phi$. For each monic irreducible $Q\in\Fq[t]$, for each $\sigma\in\Gal(K^{\sep}/K)$, and for each sequence $\{x_i\}_{i\in\N}\subset K^{\sep}$ such that $\Phi_{Q}(x_{i+1})=x_i$ for each $i\in\N$ while $\Phi_Q(x_1)=x$, we consider the map 
$$\sigma\mapsto \{\sigma(x_i)-x_i\}_{i\in\N}.$$
This yields another continuous representations
$$\Psi_Q:\Gal(K^{\sep}/K)\lra \T_Q$$
and more generally
$$\Psi:\Gal(K^{\sep}/K)\lra \T.$$
Furthermore, we have the following Galois action on the entire  backward orbit of $x$ under the action of $\Phi$
$$\tilde{\Psi}:\Gal(K^{\sep}/K)\lra \T\rtimes \GL_r(\T).$$
H\"{a}berli \cite{Haberli} proved that the image of $\tilde{\Psi}$ is open, again assuming that  $\Phi$ has no complex multiplication.  The assumption regarding the endomorphism ring for $\Phi$ is crucial; however Pink \cite{Pink-Kummer} proved an appropriately modified statement when $\Phi$ has complex multiplication. In particular, H\"{a}berli's result \cite{Haberli} yields that there exists a number $d:=d(x)\in\N$ (bounded above by the index of the image of ${\tilde \Psi}$ in $\T\rtimes \GL_r(\T)$) such that for each nonzero $Q\in\Fq[t]$, there are at most $d$ distinct Galois orbits containing all the preimages of $x$ under $\Phi_{Q}$. Hence, we have the following theorem.

\begin{thm}
\label{bounded number of orbits}
Let $\Phi:\Fq[t]\lra \End_K(\bG_a)$ be a Drinfeld module such that $\End_{K^{\sep}}(\Phi)\isomto\Fq[t]$. Then for each $\alpha\in K$ which is not torsion, there exists a number $d(\alpha)$ such that for each nonzero $Q\in\Fq[t]$ there exist at most $d(\alpha)$ distinct Galois orbits of points $y\in\Kbar$ satisfying $\Phi_Q(y)=x$. 
\end{thm}

Theorem~\ref{bounded number of orbits} allows us to complete the proof of Theorem~\ref{backward orbit result}.

\begin{proof}[Proof of Theorem~\ref{backward orbit result}.]
Our proof follows the strategy from \cite[Theorems 2.5 and 2.6]{Vijay}. For each nonzero $Q\in\Fq[t]$ we let $d_Q(\alpha)$ be the number of Galois orbits contained in $\Phi_Q^{-1}(\alpha)$. Then $d_{Q_1}\le d_{Q_2}$ whenever $Q_1\mid Q_2$, and  $d_{Q_1Q_2}=d_{Q_1}d_{Q_2}$ whenever $\gcd(Q_1,Q_2)=1$. Indeed, since $\gcd(Q_1,Q_2)=1$ there exist $R_1,R_2\in\Fq[t]$ such that $R_1Q_1+R_2Q_2=1$. Then for each pair $(\delta_1,\delta_2)\in \Phi_{Q_1}^{-1}(\alpha)\times \Phi_{Q_2}^{-1}(\alpha)$, we let $\delta_{1,2}:=\Phi_{R_2}(\delta_1)+\Phi_{R_1}(\delta_2)\in\Phi_{Q_1Q_2}^{-1}(\alpha)$. Furthermore, if $(\delta_1',\delta_2')$ is another such pair such that $\delta_i'$ is not Galois conjugate with $\delta_i$ for some $i\in\{1,2\}$, then $\delta_{1,2}'$ is not Galois conjugate with $\delta_{1,2}$. Indeed, if there is some $\sigma\in\Gal(\Kbar/K)$ such that $\delta_{1,2}'=\delta_{1,2}^\sigma$, then also $\Phi_{Q_2}(\delta_{1,2}')=\Phi_{Q_2}(\delta_{1,2}^\sigma)$ and thus
$$\Phi_{Q_2R_2}(\delta_1')+\Phi_{R_1}(\alpha)=\Phi_{Q_2R_2}(\delta_1^\sigma)+\Phi_{R_1}(\alpha),$$
which yields that $\delta_1'-\delta_1^\sigma\in \Phi[Q_2R_2]$. On the other hand, $\delta_1'-\delta_1^\sigma\in \Phi[Q_1]$ (since both are in $\Phi_{Q_1}^{-1}(\alpha)$). Because $\gcd(Q_1,Q_2R_2)=1$ we conclude that $\delta_1'=\delta_1^\sigma$; similarly we get that $\delta_2'=\delta_2^\sigma$ which yields that the pairs of points $(\delta_1',\delta_2')$ and $(\delta_1,\delta_2)$ are Galois conjugate, contrary to our assumption.

So indeed, $d_{Q_1Q_2}=d_{Q_1}d_{Q_2}$ if $\gcd(Q_1,Q_2)=1$. By Theorem~\ref{bounded number of orbits}, $d_Q$ is bounded above independently of $Q$, which yields that there exists a nonzero $P:=P(\alpha)\in\Fq[t]$ such that for all $Q\in\Fq[t]$ we have $d_Q(\alpha)=d_R(\alpha)$, where $R=\gcd(P,Q)$. In particular, with the above notation, if $\delta\in\Phi_R^{-1}(\alpha)$ then all points in $\Phi_{Q/R}^{-1}(\delta)$ are Galois conjugates.

At the expense of replacing $S$ by a larger set  we may assume it contains all places where either $\beta$ or $\alpha$ is not a unit (note that neither $\alpha$ nor $\beta$ are equal to $0$ since they are nontorsion).  Thus for each $v\notin S$ we have that each point in $\OO_\Phi^-(\alpha)$ is also a $v$-adic unit. Therefore a point $\gamma\in\OO_\Phi^-(\alpha)$ is $S$-integral with respect to $\beta$ if and only if for all $v\notin S$ we have $|\gamma-\beta|_v= 1$.

Now, assume $\gamma\in\Phi_Q^{-1}(\alpha)$ is $S$-integral with respect to $\beta$ for some $Q\in\Fq[t]$. 
Let $R=\gcd(P,Q)$ (where $P:=P(\alpha)\in\Fq[t]$ is defined as above for $\alpha$), and let $\delta\in \Phi^{-1}_R(\alpha)$ such that $\gamma\in \Phi_{Q/R}^{-1}(\delta)$. Because $\gamma$ is $S$-integral with respect to $\beta$, then for each conjugate $\gamma^\sigma$, and for each place $v\notin S$, we have $|\gamma^\sigma - \beta|_v = 1$. Hence taking the product over all conjugates $\gamma^\sigma$ satisfying $\Phi_{Q/R}(\gamma^\sigma)=\delta$ we obtain
$$|\Phi_{Q/R}(\beta)-\delta|_v=|\Phi_{Q/R}(\beta - \gamma)|_v=\prod_{\sigma}|\beta - \gamma^\sigma|_v=1 .$$
In the above computation we used that $\Phi_t$ is monic and therefore the leading coefficient of $\Phi_{Q/R}$ is a $v$-adic unit since it is in $\Fpbar$.
 Also, we used that
$$|\Phi_{Q/R}(\beta-\gamma)|_v=\prod_{z\in \Phi[Q/R]} |\beta-\gamma-z|_v=\prod_{\sigma} |\beta-\gamma^\sigma|_v,$$
since all points in $\Phi_{Q/R}^{-1}(\delta)$ are Galois conjugates, and therefore for each $z\in \Phi[Q/R]$ we have that $\gamma+z=\gamma^\sigma$ for some $\sigma\in\Gal(\Kbar/K)$.
 
So, letting $S(\delta)$ be the places of $K(\delta)$ which lie above the places from $S$, we conclude that $\Phi_{Q/R}(\beta)$ is $S(\delta)$-integral with respect to $\delta$ (where the ground field is now $K(\delta)$). On the other hand, since $\beta\notin\Phi_{\tor}$, \cite[Theorem~2.5]{Siegel} yields that there exist at most finitely $Q/R\in\Fq[t]$ such that $\Phi_{Q/R}(\beta)$ is $S(\delta)$-integral with respect to $\delta$. Finally, noting that there are only finitely many 
$$\delta \in \bigcup_{R\mid P}\Phi_R^{-1}(\alpha),$$
we conclude our proof.
\end{proof}

\subsection{Canonical height associated to $\Phi$}

For a point $x$ in $\Kbar$ its usual Weil height is defined as
$$h(x)=\sum_{v\in M_K}\frac{1}{[K(x):K]} \sum_{\sigma\in\Gal(\Kbar/K)} \log^+\left|x^\sigma\right|_v.$$
where by $\log^+z$ we always denote $\log\max\{z,1\}$ (for any real
number $z$).

The \emph{global canonical height} $\hhat_{\Phi}(x)$ associated to the
Drinfeld module $\Phi$ was first introduced by Denis~\cite{denis92}
(Denis defined the global canonical heights for general $T$-modules
which are higher dimensional analogue of Drinfeld modules). 
For each $x\in\Kbar$, the global canonical height is defined  as
$$\hhat_\Phi(x)=\lim_{n\to\infty} \frac{h\left(\Phi_{t^n}(x)\right)}{q^{rn}},$$
where $h$ is the usual (logarithmic) Weil height on $\Kbar$. Denis \cite{denis92} showed that $\hhat$ differs from the usual Weil height $h$ by a bounded amount, and also showed that $x\in\Phi_{\tor}$ if and only if $\hhat_{\Phi}(x)=0$.

Following Poonen \cite{Poonen} and Wang \cite{Wang},
 for each  $x\in \C_v$, the  \emph{local canonical height} of $x$ is
 defined as follows   
$$\hhat_{\Phi,v}(x):=\lim_{n\to\infty}
\frac{ \log^+|\Phi_{t^n}(x)|_v}{q^{rn}}.$$
  It is immediate that
$\hhat_{\Phi,v}(\Phi_{t^i}(x))=q^{ir} \hhat_{\Phi,v}(x)$
and thus $\hhat_{\Phi,v}(x)=0$ whenever $x\in\Phi_{\tor}$.

Now, if $f(x)=\sum_{i=0}^d a_i x^i$ is any polynomial defined over $K$, then
$|f(x)|_v=|a_dx^d|_v>|x|_v$ when $|x|_v>M_v$, where
\begin{equation}
\label{definition N_v}
M_v=M_v(f):=\max\left\{\left(\frac{1}{|a_d|}\right)^{\frac{1}{d-1}},\max\left\{
    \left|\frac{a_i}{a_d}\right|^{\frac{1}{d-i}}\right\}_{0\le
    i<d}\right\}.
\end{equation}
Moreover, for a Drinfeld module $\Phi$, if $|x|_v>M_v(\Phi_t)$
then $\hhat_{\Phi,v}(x)=\log|x|_v+\frac{\log|a_d|_v}{d-1}>0$. In the special case that $\Phi$ has good reduction at $v$, then $M_v(\Phi_t)=1$ and so, if $|x|_v>1$ then $\hhat_{\Phi,v}(x)=\log|x|_v$, while if $|x|_v\le 1$ then $\hhat_{\Phi,v}(x)=0$. 

We define the \emph{$v$-adic filled Julia set} $\cJ_v$ be the set of all $x\in \C_v$ such that $\hhat_{\Phi,v}(x)=0$. So, we know that if $x\in \cJ_v$, then $|x|_v\le M_v$. In particular, if $v$ is a place of good reduction for $\Phi$, then $\cJ_v$ is the unit disk.

As shown
in \cite{Poonen} and \cite{Wang}, the global canonical height
decomposes into a sum of the corresponding local canonical
heights, as follows 
$$\hhat_\Phi(x)=\sum_{v\in M_K}\frac{1}{[K(x):K]}\sum_{\sigma\in\Gal(\Kbar/K)} \hhat_{\Phi,v}\left(x^{\sigma}\right).$$
We note that the theory of canonical height associated
to a Drinfeld module is a special case of the canonical heights
associated to morphisms on algebraic varieties developed by
Call~and~Silverman (see~\cite{Call-Silverman} for details). The definition for the canonical height functions given above
seems to depend on the particular choice of the map $\Phi_t$. On the
other hand, one can define the canonical heights
$\hhat_{\Phi}$  as in~\cite{denis92} by letting 
$$\hhat_\Phi(x)=\lim_{\deg(R)\to\infty}
\frac{h\left(\Phi_{R}(x)\right)}{q^{r\deg(R)}},$$
and similar formula for canonical local heights $\hhat_{\Phi,v}(x)$
where $R$ runs through all non-constant polynomials in $\Fq[t].$ In \cite{Poonen} and \cite{Wang} it is proven that  both definitions yield the same height
function. 

Finally, we observe that Ingram \cite{Ingram} defined the local canonical height in a slightly different way, i.e., Ingram's local canonical height $\lambda_{v}(x)$ for a point $x\in \C_v$ is defined as follows:
$$\lambda_v(x):=\hhat_{\Phi,v}(x) - \log|x|_v + c_v,$$
where $c_v:=-\log|a_r|_v/(q^r-1)$, where $a_r$ is the leading coefficient of $\Phi$. Using the product formula, we conclude that if $x\in\Kbar$ then $$\hhat_\Phi(x)=\sum_{v\in\Omega_K} \frac{1}{[K(x):K]} \sum_{\sigma\in\Gal(\Kbar/K)}\lambda_v\left(x^\sigma\right).$$ 
Furthermore, for normalized Drinfeld modules, one has
$$\lambda_v(x)=\hhat_{\Phi,v}(x)-\log|x|_v.$$
The advantage in Ingram's definition is that $\lambda_v(x):=G_v(x,0)$, where $G_v(x,y)$ is the Green's function for the dynamical system induced by the action of $\Phi$ on the affine line. 

\subsection{Points of small canonical height}

Next we give a construction showing that there is no Bogomolov-type statement of Ih's Conjecture for Drinfeld modules, i.e., there exist infinitely many points of canonical height arbitrarily small which are $S$-integral with respect to a given nontorsion point $\beta$.
\begin{example}
\label{no Bogomolov-Ih}
Indeed, let $\Phi$ be the Carlitz module corresponding to $\Phi_t(x)=tx+x^p$, where $p$ is a prime number, and let $\beta=1$. Then $\beta$ is nontorsion for $\Phi$ because $\deg_t(\Phi_{t^n}(1))=p^{n-1}$ for all $n\ge 1$. For each positive integer $n$, consider $x_n\in \Fp(t)^{\sep}$ which is a root of the equation
$$\Phi_{t^n}(z) \cdot (z-1)=1.$$
We let $S=\{v_\infty\}\subset \Omega_{\Fp(t)}$. Then it is immediate to see that $x_n$ must be $v$-integral for each $v\notin S$ (otherwise $|\Phi_{t^n}(x_n)|_v>1$ and also $|x_n-1|_v>1$, which is a contradiction). Furthermore, if $|x_n-1|_v<1$ for $v\notin S$, then $|\Phi_{t^n}(x_n)|_v>1$, which is again a contradiction since it would imply that $|x_n|_v>1$. Similarly, for each conjugate $x_n^\sigma$ we have $|x_n^\sigma-1|_v=1$ for each $v\notin S$. On the other hand, $x_n$ has height tending to $0$. Indeed, as shown by Denis \cite{denis92} there exists a positive constant $C$ such that $|h(z)-\hhat_\Phi(z)|\le C$ for all algebraic points $z$. So, 
$$p^n \hhat_\Phi(x_n)=\hhat_\Phi(\Phi_{t^n}(x_n))=\hhat_\Phi\left(\frac{1}{x_n-1}\right) \le h\left(\frac{1}{x_n-1}\right)+C=h(x_n)+C\le \hhat(x_n)+2C.$$
Hence, $\hhat_\Phi(x_n)\to 0$ as $n\to\infty$. 
\end{example}


\section{Preliminary results on torsion points}
\label{preliminary torsion}

In this Section we assume that $\Phi$ is in normal form and that it has good reduction at all finite places of $K$.

\begin{lemma}
\label{discreteness of the torsion}
Let $s$ be a real number in $(0,1)$. If $v\in\Omega_K\setminus S_\infty$, then there exist at most finitely many $x\in \Phi_{\tor}$ such that $|x|_v<s$.
\end{lemma}

\begin{proof}
Let $P\in \Fq[t]$ be the unique irreducible monic polynomial such that $|P|_v<1$, i.e. the place $v$ lies above the place coresponding to the polynomial $P$ in $\Fq(t)$. 

We first observe that if $|x|_v<1$, then for each $a\in\Fq(t)$ we have $|\Phi_a(x)|_v<s$ because each coefficient of  $\Phi_a$ is integral at $v$.
 
Secondly, we claim that if $x\in\Phi_{\tor}$ such that $|x|_v<1$, then there exists $n\in\N$ such that $\Phi_{P(t)^n}(x)=0$. Indeed, using our first observation it suffices to prove that if $|x|_v<1$ and $\Phi_{Q(t)}(x)=0$, where $Q\in\Fq[t]$ is relatively prime with $P(t)$, then $x=0$. Because $|x|_v<1$ and each coefficient of $\Phi_{Q(t)}$ is integral at $v$ while $|Q(t)|_v=1$, we conclude that $|\Phi_{Q(t)}(x)|_v=|Q(t)x|_v=|x|_v$. Hence, indeed $x=0$ as claimed.

So, if $0\ne x\in\Phi_{\tor}$ satisfies $|x|_v<s<1$ then there exists some $n\in\N$ such that $\Phi_{P(t)^n}(x)=0$. Assume $n$ is the smallest such positive integer, and let $y=\Phi_{P(t)^{n-1}}(x)$. Then $0\ne y$ and $\Phi_{P(t)}(y)=0$. Let $s_0:=|P(t)|_v^{1/(q-1)}<1$; so, if $|y|_v<s_0$, then $|\Phi_{P(t)}(y)|_v=|P(t)y|_v\ne 0$ which yields a contradiction. Therefore, $|y|_v\ge s_0$. On the other hand, if $0<|z|_v<1$ then
\begin{equation}
\label{how much we go down each time}
|\Phi_{P(t)}(z)|_v\le \max\{|P(t)z|_v, |z|_v^q\}<|z|_v
\end{equation}
since each coefficient of $\Phi_{P(t)}$ is integral at $v$. Because $|P(t)|_v<1$ and also $|z|_v< 1$, then $\left|P(t)z\right|_v<s_0$.  Thus, if $n>n_0:=1+ \log_q\left(\log_s(s_0)\right)$ inequality \eqref{how much we go down each time} yields that $|\Phi_{P(t)^{n-1}}(x)|_v=|y|_v<s_0$. This yields a contradiction with the fact that $y\ne 0$ but $\Phi_{P(t)}(y)=0$. So, in conclusion, if $x\in\Phi_{\tor}$ such that $|x|_v<s$ then $\Phi_{P(t)^{n_0}}(x)=0$, where $n_0$ is a positive integer depending only on $s$ and on $v$ (note that $s_0$ depends only on $v$). Thus there exist at most finitely many torsion points $x$ satisfying the inequality $|x|_v<s$. 
\end{proof}

Lemma~\ref{discreteness of the torsion} is a special case of \cite[Theorem 2.10]{Tate-Voloch}; however, our result is more precise since we assume each coefficient of $\Phi_t$ is integral at $v$. 
In particular, the following result is an immediate corollary.
\begin{cor}
\label{not too close to any number}
Let $z\in \Kbar$, let $v\in\Omega_K\setminus S_\infty$ and let $s$ be a real number such that $0<s<1$. Then there exist at most finitely many $x\in \Phi_{\tor}$ such that $|z-x|_v<s$.
\end{cor}

%
%
%
%
%
%
%


\section{Ih's Conjecture for Drinfeld modules}
\label{equi}

Assume $\Phi$ is in normal form and also it has everywhere good reduction away from $S_\infty$.  Also, using the notation from \eqref{definition N_v}, for each place $v\notin S_\infty$, if $|x|_v>1$ then $\hhat_{\Phi,v}(x)=\log|x|_v>0$ (since the leading coefficient of $\Phi_t$ is a $v$-adic unit). Let $\beta\in K$ be a nontorsion point. We'll prove Theorem~\ref{Ih's conjecture result} as a consequence of Theorem~\ref{alternative way of computing the canonical height}.

\begin{proof}[Proof of Theorem~\ref{Ih's conjecture result}.]
First, we enlarge $S$ so that it contains $S_\infty$; clearly enlarging $S$ can only increase the number of torsion points which are $S$-integral with respect to $\beta$. Then for all $v\notin S$ we know that for a torsion point $\gamma$ we have $|\gamma|_v\le 1$ since $\Phi$ has good reduction at $v$. Hence for each $v\notin S$, if $\gamma\in\Phi_{\tor}$ is $S$-integral with respect to $\beta$, then
$$|\beta-\gamma^\sigma|_v=\max\{|\beta|_v,1\},$$
for each $\sigma\in\Gal(K^{\sep}/K)$.

Assume there exist infinitely many torsion points $\gamma_n$ which are $S$-integral with respect to $\beta$. 
By Theorem~\ref{alternative way of computing the canonical height}, we know that
$$\hhat_\Phi(\beta)=\sum_{v\in\Omega_K}\lim_{n\to\infty}\frac{1}{[K(\gamma_n):K]}\sum_{\sigma:K(\gamma_n) \lra K^{\sep}}  \log|\beta-\gamma_n^{\sigma}|_v.$$
On the other hand, since $\gamma_n$ is $S$-integral with respect to $\beta$, we know that
$$\log|\beta-\gamma_n^{\sigma}|_v=\log^+|\beta|_v.$$
Hence the above outer sum consists of only finitely many nonzero terms and therefore we may reverse the order of the summation with the limit, and conclude that
$$\hhat_\Phi(\beta)=\lim_{n\to\infty}\sum_{v\in\Omega_K}\frac{1}{[K(\gamma_n):K]}\sum_{\sigma:K(\gamma_n) \lra K^{\sep}}  \log|\beta-\gamma_n^{\sigma}|_v=0,$$
by the product formula applied to each $\beta-\gamma_n$ (note that this element is nonzero since $\beta\notin\Phi_{\tor}$). Thus we obtain that $\hhat_\Phi(\beta)=0$, which contradicts the fact that $\beta\notin\Phi_{\tor}$. So, indeed there are at most finitely many torsion points which are $S$-integral with respect to $\beta$.
\end{proof}

We are left to proving Theorem~\ref{alternative way of computing the canonical height}. This will follow from the following result.
\begin{thm}
\label{alternative way of computing the local height}
Let $\beta\in K$, let $v\in M_K$, and let $\{\gamma_n\}\subset\Kbar$ be an infinite sequence of torsion points for the Drinfeld module $\Phi$. Then
$$\hhat_{\Phi,v}(\beta)=\lim_{n\to\infty} \frac{1}{[K(\gamma_n):K]}\cdot \sum_{\sigma\in\Gal(\Kbar/K)}\left|\beta-\gamma_n^\sigma\right|_v.$$
\end{thm}

We'll prove Theorem~\ref{alternative way of computing the local height} by analyzing two cases depending on whether $v\in S_\infty$ or not.

\begin{prop}
\label{notin S infinity}
Theorem~\ref{alternative way of computing the local height} holds if $v\notin S_{\infty}$.
\end{prop}

\begin{proof}
Firstly, since $v\notin S_\infty$, then $v$ is a place of good reduction for $\Phi$ and then $\cJ_v$ consists of the unit disk. We denote by $J_v$ the Julia set, which is the boundary of $\cJ_v$. Note that since $\cJ_v$ is closed, then $J_v\subset \cJ_v$; finally, because $v$ is a place of good reduction and so, $\cJ_v$ is the unit disk, then $J_v$ is its boundary and thus consists of all $z\in\C_v$ such that $|z|_v=1$.

There are two cases: either $\beta\in J_v$, or not. 

{\bf Case 1.} Assume $\beta\in J_v$.

Then $\hhat_{\Phi,v}(\beta)=0$, since in particular $\beta\in\cJ_v$. If for each torsion point $\gamma$ we have that $|\beta -\gamma|_v=1$, then clearly
$$\lim_{n\to\infty}  \frac{1}{[K(\gamma_n):K]}\cdot \sum_{\sigma\in\Gal(\Kbar/K)}\left|\beta-\gamma_n^\sigma\right|_v=0=\hhat_{\Phi,v}(\beta).$$
Now, if there exists some torsion point $\gamma$ such that $|\beta-\gamma|_v<1$, let $s$ be any real number satisfying 
$$|\beta-\gamma|_v<s<1.$$ 
By Lemma~\ref{not too close to any number} we conclude that there exist finitely many torsion points $\gamma '$ such that $|\beta-\gamma '|_v<s$. In particular, for all $n$ sufficiently large, and for all $\sigma\in\Gal(\Kbar/K)$, we have  $|\beta-\gamma_n^\sigma|_v\ge s$ and thus
$$\frac{1}{[K(\gamma_n):K]}\cdot \sum_{\sigma} \log|\beta-\gamma_n^\sigma|_v\ge \log(s)$$
and so, letting $s\to 1$ we obtain
$$\lim_{n\to\infty} \frac{1}{[K(\gamma_n):K]}\cdot \sum_{\sigma} \log|\beta-\gamma_n^\sigma|_v\ge 0.$$
On the other hand, since each torsion point $\gamma '$ is in $\cJ_v$, we have $|\gamma '|_v\le 1$. Because $|\beta|_v=1$ (because $\beta\in J_v$) we conclude that $|\beta-\gamma '|_v\le 1$ for all $\gamma '\in\Phi_{\tor}$. Hence 
$$\lim_{n\to\infty} \frac{1}{[K(\gamma_n):K]}\cdot \sum_{\sigma:K(\gamma_n)\lra \Kbar} \log|\beta-\gamma_n^\sigma|_v\le 0$$
and therefore, in conclusion
$$\lim_{n\to\infty} \frac{1}{[K(\gamma_n):K]}\cdot \sum_{\sigma\in\Gal(\Kbar/K)} \log|\beta-\gamma_n^\sigma|_v= 0=\hhat_{\Phi,v}(\beta).$$

{\bf Case 2.} Assume $\beta\notin J_v$.

In this case, the function $z\mapsto \log|\beta-z|_v$ is continuous on $J_v$ and the equidistribution result of Baker-Hsia \cite[Proposition 4.5]{Baker-Hsia} yields that 
$$\lim_{n\to\infty} \frac{1}{[K(\gamma_n):K]}\cdot \sum_{\sigma\in\Gal(\Kbar/K)} \log|\beta-\gamma_n^\sigma|_v = \int_{J_v} \log|z-x| d\mu_v(x),$$
where $\mu_v$ is the invariant measure on the Julia set. 
So, if $|z|_v>1$, then the above limit equals $\log|z|_v$, while if $|z|_v<1$, then the above limit equals $0$. In both cases, the limit is the local canonical height of $z$ at $v$.
\end{proof}

\begin{remark}
\label{why we need the hypotesis}
The method of proof for Proposition~\ref{notin S infinity} reveals the necessity for our hypothesis from Theorem~\ref{Ih's conjecture result} that $\Phi$ has good reduction at all finite places. Indeed, this allows us to conclude that $\cJ_v$ is the unit disk in $\C_v$ and moreover that for each $z\in C_v$ we have $\hhat_{\Phi,v}(z)=\log^+|z|_v$. It is likely that Proposition~\ref{notin S infinity} holds without the above hypothesis but one would need a different approach.
\end{remark}

Assume now that $v\in S_\infty$. In this case the setup is from our paper \cite{equi}. As shown by Theorem $4.6.9$ of \cite{Goss}, there exists an $\Fq[t]$-lattice $\Lambda_v\subset\C_v$ associated to the generic characteristic Drinfeld module $\Phi$; $\C_v$ is the completion of $\bar{K_v}$ which is a complete, algebraically closed field. Let $\e$ be the exponential function defined in $4.2.3$ of \cite{Goss} which gives a continuous (in the $v$-adic topology) isomorphism $$\e:\C_v/\Lambda\rightarrow\C_v.$$
The torsion submodule of $\Phi$ in $\C_v$ is isomorphic naturally through $\e^{-1}$ to 
$$\left(\Fq(t)\tensor_{\Fq[t]}\Lambda_v\right)/\Lambda_v.$$ 
Since each torsion point is repelling, we conclude that the corresponding Julia set $J_v$ is the closure of the set of all torsion points. We also note that the completion of $\Fq(t)$ with respect to the restriction of $v$ on $\Fq(t)$ is $\Fq\left(\left(\frac{1}{t}\right)\right)$.

Then the restriction of $\e$ on $\left(\Fq\left(\left(\frac{1}{t}\right)\right)\tensor_{\Fq[t]}\Lambda_v\right)/\Lambda_v$ gives an isomorphism between $\left(\Fq\left(\left(\frac{1}{t}\right)\right)\tensor_{\Fq[t]}\Lambda_v\right)/\Lambda_v$ and $J_v$.

Let $r$ be the rank of $\Lambda_v$ which is the same as the rank of the Drinfeld module $\Phi$. Then $\left(\Fq(t)\tensor_{\Fq[t]}\Lambda_v\right)/\Lambda_v\isomto (\Fq(t)/\Fq[t])^r$. Let $\omega_1,\dots,\omega_r$ be a fixed $\Fq[t]$-basis of $\Lambda_v$. Furthermore, we may assume the $\omega_i$'s is a basis of \emph{successive minima} as defined by Taguchi \cite{Taguchi}, i.e., for each $P_1,\dots,P_r\in \Fq[t]$ we have
$$|P_1(t)\omega_1+\cdots +P_r(t)\omega_r|_v = \max_{i=1}^r |P_i(t)\omega_i|_v.$$
Using Proposition $4.6.3$ of \cite{Goss}, $\left(\Fq\left(\left(\frac{1}{t}\right)\right)\tensor_{\Fq[t]}\Lambda_v\right)/\Lambda_v$ is isomorphic to $\left(\mathbb{F}_q((\frac{1}{t}))/\mathbb{F}_q[t]\right)^r$. Then we have the isomorphism $$\E:\left(\mathbb{F}_q\left(\left(\frac{1}{t}\right)\right)/\mathbb{F}_q[t]\right)^r\rightarrow J_v\text{ given by}$$
$$\E(\gamma_1,\dots,\gamma_r):=\e(\gamma_1\omega_1+\dots+\gamma_r\omega_r)\text{, 
for each $\gamma_1,\dots,\gamma_r\in\mathbb{F}_q\left(\left(\frac{1}{t}\right)\right)/\mathbb{F}_q[t]$.}$$
We construct the following  group isomorphism $$\tau:\mathbb{F}_q\left(\left(\frac{1}{t}\right)\right)/\mathbb{F}_q[t] \rightarrow\frac{1}{t}\cdot\mathbb{F}_q\left(\left(\frac{1}{t}\right)\right)\text{, given by}$$
\begin{equation}
\label{E:sigma}
\tau\left(\sum_{i\ge -n}\alpha_i\left(\frac{1}{t}\right)^i\right)=\sum_{i\ge 1}\alpha_i\left(\frac{1}{t}\right)^i,
\end{equation}
for every natural number $n$ and for every $\sum_{i\ge -n}\alpha_i\left(\frac{1}{t}\right)^i\in\mathbb{F}_q\left(\left(\frac{1}{t}\right)\right)$ (obviously, $\tau$ vanishes on $\mathbb{F}_q[t]$). The group $\frac{1}{t}\cdot\mathbb{F}_q[[\frac{1}{t}]]$ is a topological group with respect to the restriction of $v$ on $\frac{1}{t}\cdot\mathbb{F}_q[[\frac{1}{t}]]$. Hence, the isomorphism $\tau^{-1}$ induces a topological group structure on $\mathbb{F}_q\left(\left(\frac{1}{t}\right)\right)/\mathbb{F}_q[t]$. Therefore, $\tau$ becomes a continuous isomorphism of topological groups.
We endow $\left(\mathbb{F}_q\left(\left(\frac{1}{t}\right)\right)/\mathbb{F}_q[t]\right)^r$ with the corresponding product topology. The isomorphism $\tau$ extends diagonally to another continuous isomorphism, which we also call $$\tau:\left(\mathbb{F}_q\left(\left(\frac{1}{t}\right)\right)/\mathbb{F}_q[t]\right)^r \rightarrow\left(\frac{1}{t}\mathbb{F}_q\left[\left[\frac{1}{t}\right]\right]\right)^r=:G.$$ Moreover, using that $\e$ is a continuous morphism, we conclude 
\begin{equation}
\label{E:continuous isomorphism}
\E\tau^{-1}:G\rightarrow J_v\text{ is a continuous isomorphism.}
\end{equation} 
Since $\omega_1,\dots,\omega_r$ is a basis of $\Lambda$ formed by successive minima,  for each $a_1,\dots,a_r\in  \Fq[[1/t]]$ we have
$$|\sum_{i=1}^r a_i(1/t)\omega_i |_v = \max_{i=1}^r |a_i(1/t)\omega_i|_v.$$
Indeed, without loss of generality, we assume $a_i\ne 0$ for $i=1,\dots, s$ and $a_i=0$ for $i>s$ (for some $s\le r$). Let $N\in\N$ such that $|1/t^N|_v = \min_{i=1}^s |a_i(1/t)|_v$. Then there exist nonzero $P_1,\dots, P_s\in\Fq[t]$ of degree at most equal to  $N$  such that for each $i=1,\dots, s$ we have
$$a_i(1/t) = P_i(t)/t^N + b_i(1/t),$$
where each $b_i(1/t)\in 1/t^{N+1}\cdot \Fq[[1/t]]$. Then
$$\left|\sum_{i=1}^s b_i(1/t)\omega_i\right|_v \le |1/t^{N+1}|_v\cdot \max_{i=1}^s |\omega_i|_v<|1/t^N|_v\cdot \max_{i=1}^s |\omega_i|_v.$$
On the other hand,
$$\left|\frac{\sum_{i=1}^s P_i(t)\omega_i}{t^N}\right| = \frac{\max_{i=1}^s |P_i(t)\omega_i|_v}{|t^N|_v} \ge \frac{\max_{i=1}^s |\omega_i|_v}{|t^N|_v},$$
since $|P_i(t)|_v\ge 1$ because each $P_i$ is nonzero. Therefore,
$$\left|\sum_{i=1}^r a_i(1/t)\omega_i\right|_v = \frac{\max_{i=1}^s |P_i(t)\omega_i|_v}{|t^N|_v} = \max_{i=1}^r |a_i(1/t)\omega_i|_v.$$

\begin{notation}
\label{N:D2}
Let $\nu_v$ be the Haar measure on $G$, normalized so that its total mass is $1$. Let $\mu_v:=\left(\E\sigma^{-1}\right)_{*}\nu_v$ be the induced measure on $J_v$ (i.e. $\mu_v(V):=\nu_v\left(\sigma\E^{-1}(V)\right)$ for every measurable $V\subset J_v$). 
\end{notation}
Because $\nu_v$ is a probability measure, then $\mu_v$ is also a probability measure. Because $\nu_v$ is a Haar measure on $G$ and $\E\sigma^{-1}$ is a group ismorphism, then $\mu_v$ is a Haar measure on $J_v$.

\begin{defi}
\label{D:Dirac measure}
Given $x\in K^{\sep}$, we define a probability measure $\dirac_{x}$ on $\C_v$ by
$$\dirac_x=\frac{1}{[K(x):K]}\sum_{\sigma\in\Gal(\Kbar/K)}\delta_{x^\sigma},$$
where $\delta_y$ is the Dirac measure on $\C_v$ supported on $\{y\}$.
\end{defi}

Before we can state the equidistribution result from \cite[Theorem 2.7]{equi} (see our Theorem~\ref{T:equi}), we need to define the concept of weak convergence for a sequence of probability measures on a metric space.
\begin{defi}
\label{D:weak convergence}
A sequence $\{\lambda_k\}$ of probability measures on a metric space $S$ weakly converges to $\lambda$ if for any bounded continuous function $f:S\rightarrow\mathbb{R}$, $\int_S f d\lambda_k\rightarrow\int_S f d\lambda$ as $k\rightarrow\infty$. In this case we use the notation $\lambda_k\goes\lambda$.
\end{defi}

\begin{thm}
\label{T:equi}
Let $\Phi:A\rightarrow K\{\tau\}$ be a Drinfeld module of generic characteristic such that $\End_{K^{\sep}}(\Phi)\isomto \Fq[t]$. Let $\{x_k\}$ be a sequence of distinct torsion points in $\Phi$. Then $\dirac_{x_k}\goes\mu_v$.
\end{thm}

\begin{remark}
\label{important observation}
Theorem~\ref{T:equi} is stated in \cite[Theorem 2.7]{equi} when $K$ is a transcendental extension of $\Fq(t)$ since the author needed to use the fact that the image of the Galois group is open in the ad\`elic Tate module for a Drinfeld module, and at that moment the result was only known under the assumption that $K$ is transdendental over $\Fq(t)$ (see \cite{Breuer-Pink}). However, since then R\"utsche and Pink \cite{Pink-Rutsche} removed the assumption on $K$, and thus Theorem~\ref{T:equi} holds in the above generality (see also \cite[Remarks 3.2]{equi}).
\end{remark}

Furthermore the proof from \cite{equi} yields that the points in $\Phi[Q]$ are equidistributed in the Julia set $J_v$ with respect to $d\mu_v$ as $\deg(Q)\to\infty$. This follows as the main result of \cite{equi} using the fact that $J_v$ is isomorphic (as a topological group) to $G$ and the points in $\Phi[Q]$ correspond in $G$ to all points of the form
$$\left(\frac{P_1}{Q},\dots,\frac{P_r}{Q}\right)$$
where the monic polynomials $P_i$ have degrees less than $\deg(Q)$. More precisely, for a generic open subset of $\left(\frac{1}{t}\cdot \Fq\left[\left[\frac{1}{t}\right]\right]\right)^r$ which is of the form (see \cite[page 847, equation (6)]{equi})
$$U:=\left(a_1\left(\frac{1}{t}\right),\cdots ,a_r\left(\frac{1}{t}\right)\right)+ \left(\frac{1}{t^{n_1+1}}\cdot \Fq\left[\left[\frac{1}{t}\right]\right],\cdots ,\frac{1}{t^{n_r+1}}\cdot \Fq\left[\left[\frac{1}{t}\right]\right]\right)$$
for some polynomials $a_i\in\frac{1}{t}\cdot \Fq\left[\left[\frac{1}{t}\right]\right]$ of degree at most $n_i$, 
we have to show (see \cite[page 847, equation (7)]{equi}) that the number of tuples
$$\left(\frac{P_1}{Q},\cdots, \frac{P_r}{Q}\right)\in U$$
where $\deg(P_i)<\deg(Q)=d$ is asymptotic to 
$$q^{dr-\sum_{i=1}^r n_i}\text{ as }d\to\infty.$$
As argued in \cite{equi}, it suffices to prove this claim when $r=1$, in which case the above statement reduces to show that (as $d\to\infty$) there are $q^{d-n_1}$ distinct polynomials $P_1$ of degree less than $d$ satisfying
$$\frac{P_1}{Q}-a_1\left(\frac{1}{t}\right)\in \frac{1}{t^{n_1+1}}\cdot \Fq\left[\left[\frac{1}{t}\right]\right].$$
This last statement follows at once since this last condition induces $n_1$ conditions on the $d$ coefficients of $P_1$.

Furthermore, a strong equidistribution result for torsion points is obtained in the proof of the main result from \cite{equi} (a similar result was proven in a more general context by Favre and Rivera-Letelier \cite{F-R-L}).
\begin{thm}
\label{strong equidistribution}
Given $\gamma\in\Phi_{\tor}$, and also given an open subset $\E\tau^{-1}(U)$ of $J_v$, where $U\subset G$ is defined as above:
$$U:=\left(a_1\left(\frac{1}{t}\right),\cdots ,a_r\left(\frac{1}{t}\right)\right)+ \left(\frac{1}{t^{n_1+1}}\cdot \Fq\left[\left[\frac{1}{t}\right]\right],\cdots ,\frac{1}{t^{n_r+1}}\cdot \Fq\left[\left[\frac{1}{t}\right]\right]\right),$$
we let
$$N(\gamma,U):=\{\sigma\in\Gal(K^{\sep}/K)\text{ : }\gamma^\sigma\in \E\tau^{-1}(U)\}.$$
Let $\delta$ be a real number in the interval $(0,1)$. Then for all $\gamma\in\Phi_{\tor}$ and for all open subsets $U$ as above, 
$$\frac{N(\gamma, U)}{[K(\gamma):K]}=\mu_v(U)+O_\delta\left([K(\gamma):K]^{-\delta}\right).$$
\end{thm}

\begin{proof}
This is the Drinfeld module analogue of the strong equidistribution result from \cite[Proposition~2.4]{Ih}, and it is essentially proven in \cite{equi} when deriving formula (7), page 847. One simply needs to be more careful when estimating the error term in \cite[(39), page 854]{equi} since this time we don't fix the open set $U$. The differences are as follows. In \cite[(33), page 852]{equi}, one estimates the number of all polynomials $q_i'$ (for $i=1,\dots, r$) to be 
$$q^{\sum_{i=1}^r \deg(b)-\deg(b')-\deg(d)-n_i}+O(1),$$
where $O(1)$ is independent of all previsously defined quantities. Then  the error term in \cite[(39), page 854]{equi} is bounded above by the number of divisors of the polynomial $b\in\Fq[t]$ (which is the order of $\gamma$). Finally, noting that the number of divisors of $b$ is bounded above by $\deg(b)$, and that \cite[(44), page 855]{equi} is bounded below (see also \cite[Remarks~3.2, page 856]{equi}) by the number of polynomials relatively prime with $b$ and of degree less than $\deg(b)$, and this  number is larger than $q^{(1-\epsilon)\deg(b)}$, for any positive real number $\epsilon$, we obtain the conclusion of Proposition~\ref{strong equidistribution}. 
\end{proof}
The following result is a simple consequence of Theorem~\ref{strong equidistribution} since $\E\tau^{-1}$ induces a local isometry between $G$ and $U$, and therefore there exists a positive real number $r_v$ such that for all subsets $U$ of $J_v$ of diameter at most $r_v$ (i.e., for each $x,y\in V$ we have $\log|x-y|_v \le r_v$), $\E\tau^{-1}$ is a distance-preserving isomorphism between $U$ and $\E\tau^{-1}(U)$.
\begin{cor}
\label{corollary strong equidistribution}
Let $\delta \in (0,1)$ and let $\beta\in J_v$. Then for all open subsets $U\subset J_v$ containing $\beta$ and of diameter at most $r_v$, and for all $\gamma\in\Phi_{\tor}$ we have
$$\frac{\# \{\sigma\in\Gal(K^{\sep}/K) \text{ : }\gamma^{\sigma} \in U\}}{[K(\gamma):K]} = \mu_v(U) + O_\delta\left([K(\gamma):K]^{-\delta}\right).$$
\end{cor}

The following result follows from the powerful lower bound for linear forms in logarithms for Drinfeld modules established by Bosser \cite{Bosser}.
\begin{fact}
\label{Bosser}
Assume $\infty\in \Omega_K$ is an infinite place. Let $\beta\in K$ be a nontorsion point and let $\gamma\in \Phi[Q]$ where $Q\in \Fq[t]$ is a monic polynomial of degree $d$. Then there exist
(negative) constants $C_0$ and $C_1$ (depending only on $\Phi$ and
$\beta$) such that
$$
\log | \gamma-\beta|_{\infty} \geq C_0 + C_1 d\log d.
  $$
\end{fact}

\begin{proof}
In \cite[Fact 3.1]{siegel}, Tucker and the author showed that Bosser's result yields the existence of some (negative) constants $C_2$ and $C_3$ such that for all polynomials $P\in\Fq[t]$ we have
\begin{equation}
\label{first bound for linear forms in logarithms}
\log|\Phi_P(\beta)|_\infty \ge C_2 +C_3\deg(P)\log\deg(P).
\end{equation}
On the other hand, if $|y|_\infty$ is sufficiently small but positive, then
$$\log |\Phi_t(y)|_\infty =\log |ty|_\infty= \log|y|_\infty +|t|_\infty.$$
Note that $|t|_\infty>0$ since $\infty$ is an infinite place. So assuming that $d$ is sufficiently large, say $d\ge d_0\ge 3$, if
$$\log |\beta-\gamma|_\infty < C_2+(C_3-1)d\log d$$
then
$$\log |\Phi_Q(\beta-\gamma)|_\infty =\log|\Phi_Q(\beta)|_\infty=\log|Q\beta|_\infty=d\log|t|_\infty +\log|\beta|_\infty< C_2+C_3 d\log d$$
contradicting thus \eqref{first bound for linear forms in logarithms}. Therefore for all $d\ge d_0$ we have that
$$\log |\beta-\gamma|_\infty \ge  C_2+(C_3-1)d\log d.$$
Since $\beta\notin\Phi_{\tor}$ we conclude that there exists $C_4<0$ such that for all torsion points $\gamma\in\Phi[Q]$ for some monic polynomial $Q$ of degree less than $d_0$ we  have
$$\log|\beta-\gamma|_\infty \ge C_4.$$
In conclusion, Fact~\ref{Bosser} holds with $C_0:=\min\{C_2,C_4\}$ and $C_1:=C_3-1$.
\end{proof}

\begin{prop}
\label{in S infinity}
Theorem~\ref{alternative way of computing the local height} holds if $v\in S_{\infty}$.
\end{prop}

\begin{proof}
Again we'll split our analysis into two cases depending on whether $\beta$ is in the Julia set $J_v$ or not.

{\bf Case 1.} Assume $\beta\notin J_v$.

As previously discussed, if $\beta\notin J_v$, then $f(z)=\log|z-\beta|_v$ is a continuous function on $J_v$ and therefore  using the result of \cite[Theorem 2.7]{equi} (see our Theorem~\ref{T:equi} above, or alternatively use \cite[Proposition 4.5]{Baker-Hsia})
$$\lim_{n\to\infty} \frac{1}{[K(\gamma_n):K]}\cdot \sum_{\sigma\in\Gal(\Kbar/K)} \log|\beta-\gamma_n^\sigma|_v = \int_{J_v} \log|\beta-z| d\mu_v(z).$$

Because $\deg(Q_n)\to \infty$ (since the torsion points $\gamma_n$ are distinct), we know that $\{\Phi[Q_n]\}_n$ is equidistributed in $J_v$ and thus
$$\int_{J_v} \log|\beta-z| d\mu_v(z)=\lim_{\deg(Q)\to\infty} \frac{1}{q^{r\deg(Q)}}\sum_{\Phi_Q(z)=0}\log|\beta-z|_v =\lim_{\deg(Q)\to\infty}\frac{\log|\Phi_Q(\beta)|_v}{q^{r\deg(Q)}}.$$
By \cite[Corollary 3.13]{newdrin} we conclude that
$$\lim_{\deg(Q)\to\infty}\frac{\log|\Phi_Q(z)|_v}{q^{r\deg(Q)}}=\hhat_{\Phi,v}(z)$$
which yields that indeed
$$\lim_{n\to\infty} \frac{1}{[K(\gamma_n):K]}\cdot \sum_{\sigma\in\Gal(\Kbar/K)} \log|\beta-\gamma_n^\sigma|_v =\hhat_{\Phi,v}(\beta).$$

{\bf Case 2.} Assume $\beta\in J_v$.

First we note that in this case $\hhat_{\Phi,v}(\beta)=0$. We need to show that
$$\lim_{n\to\infty} \frac{1}{[K(\gamma_n):K]}\cdot \sum_{\sigma\in\Gal(\Kbar/K)} \log|\beta-\gamma_n^\sigma|_v = 0 = \hhat_{\Phi,v}(\beta).$$

For each $n\in\N$,  let $Q_n\in\Fq[t]$ be the monic polynomial of minimal degree $d_n$ such that $\Phi_{Q_n}(\gamma_n)=0$.  Then we know that $e_n=[K(\gamma_n):K]>> q^{rd_n}$ since $\frac{\# \GL_r(\Fq[t]/(Q_n))}{[K(\gamma_n):K]}$ is bounded above by \cite{Pink-Rutsche}.

We claim that $\int_{J_v} \log|z|_v d\mu_v=0$. Indeed, for each point $z$ we have
$$\log|z|_v=\hhat_{\Phi,v}(z) - \lambda_v(z) + c_v(\Phi),$$
where $\lambda_v$ is the local height as defined by Ingram \cite{Ingram} and $$c_v(\Phi):=\frac{-\log|a_r|_v}{q^r-1}, $$
where $a_r$ is the leading coefficient of $\Phi_t(x)$ (for more details see \cite{Ingram}). Since we assumed $a_r=1$, then for each $z\in J_v$ we have
$$\log|z|_v= -\lambda_v(z)$$
because $\hhat_{\Phi,v}(z)=0$. However $\lambda_v(z)=g_{\mu_v}(z,0)$ where $g_{\mu_v}(x,y)$ is the Arakelov-Green function as defined in \cite[Section 10.2]{Baker-Rumely}. Therefore, by \cite[Proposition 10.12]{Baker-Rumely} we have
$$\int_{J_v}\lambda_v(z)d\mu_v=0$$
since the invariant measure $\mu_v$ is supported on the Julia set $J_v$. Thus indeed
\begin{equation}
\label{integral is zero}
\int_{J_v} \log|z|_v d\mu_v =0.
\end{equation}

Next we employ the strategy of proof from \cite{Ih}. 
Let $\epsilon:=|t^m|_v^{-1}< 1$ be sufficiently small such that the analytic map $\e: \C_\infty\lra \C_\infty$ given by
$$\e(u):=u\cdot \prod_{\omega\in\Lambda\setminus\{0\}} \left(1-\frac{u}{\omega}\right)$$
is an isomorphism restricted on the closed ball $\bar{D}(0,\epsilon)$. 

We consider 
$$J_{v,\beta,\epsilon}:=\{z\in J_v\text{ : }|z-\beta|_v\le \epsilon\}.$$
Also, we define $h_{\beta,\epsilon}:J_v\lra \R$ as follows
$$h_{\beta,\epsilon}(z):=\min\left\{0, \log\left(\frac{|z-\beta|_v}{\epsilon}\right) \right\}.$$
Then $h_{\beta,\epsilon}$ is supported on $J_{v,\beta,\epsilon}$ and it has a logarithmic singularity  at $\beta$. So, there exists a continuous function $g_{\beta,\epsilon}:J_v\lra \R$ such that $\log|z-\beta|_v = g_{\beta,\epsilon}(z)+h_{\beta,\epsilon}(z)$. 

For each $n\in\N$ we define $\mu_n:=\dirac_{\gamma_n}$ be the probability measure on $J_v$ supported on the Galois orbit of $\gamma_n$ (see Definition~\ref{D:Dirac measure}). Then for each continuous function $f:J_v\lra \R$ we have
$$\int_{J_v} f d\mu_n :=\frac{1}{[K(\gamma_n):K]}\cdot  \sum_{\sigma \in\Gal(\Kbar/K)} f\left(\gamma_n^\sigma\right).$$

\begin{lemma}
\label{integral of h}
$$\int_{J_v} h_{\beta,\epsilon} d\mu_v = -\frac{\epsilon^r}{q^r-1}.$$
\end{lemma}

\begin{proof}
Since we know that $h_{\beta,\epsilon}$ is supported on $J_{v,\beta,\epsilon}$ and moreover we know that $J_v$ is closed under translations (and $\mu_v$ is translation-invariant) it suffices to prove 
$$\int_{J_{v,0,\epsilon}} h_{0,\epsilon} d\mu_v = -\frac{\epsilon^r}{q^r-1},$$
i.e., we may assume $\beta=0$. So, we have to prove that
$$\int_{J_{v,0,\epsilon}} \log\left(\frac{|z|_v}{\epsilon}\right) d\mu_v = -\frac{\epsilon^r}{q^r-1}.$$
By our assumption, $\e$ induces an isometric analytic automorphism of $\bar{D}(0,\epsilon)$; so using the change of variables $z=\e(u)$ we compute
\begin{eqnarray*}
\int_{J_{v,0,\epsilon}} \log\left(\frac{|z|_v}{\epsilon}\right) d\mu_v\\
= \int_{\e^{-1}(J_{v,0,\epsilon})} \log \left(\frac{|\e(u)|_v}{\epsilon}\right) d\nu\\
= \int_{\e^{-1}(J_{v,0,\epsilon})} \log \left(\frac{|u|_v}{\epsilon}\right) d\nu,
\end{eqnarray*}
where $\nu$ is the measure on $\e^{-1}(J_{v,0,\epsilon})$ which is isomorphic to $$\left(\Fq((1/t))\otimes_{\Fq[t]}\Lambda\right)\cap B(0,\epsilon) = \oplus_{i=1}^r B\left(0,\epsilon |\omega_i|_v^{-1}\right)\cdot \omega_i.$$  
Furthermore, we recall that for $x_1,\dots,x_r\in \Fq[[1/t]]$ we have that 
$$\log|x_1\omega_1+\cdots +x_r\omega_r|_v = \max_{i=1}^r |x_i\omega_i|.$$
Making another change of variables $x_i=\frac{u_i}{\omega_i\cdot t^m}$ and noting that $\epsilon =|1/t^m|_v$ and that the measure $\nu$ has mass equal to $1$, we are left to show that
$$I:=\int_{\left(\Fq[[1/t]]\right)^r} \log\max_{i=1}^r |x_i|_v d\nu(x_1)\cdots d\nu(x_r) = -\frac{1}{q^r-1}.$$
Indeed, we let
$$S_1:=\left(\Fq[[1/t]]\right)^r\setminus \left(1/t\cdot \Fq[[1/t]]\right)^r.$$ 
We note that $\max_{i=1}^r\log|x_i|_v$ restricted on $S_1$ equals $0$, and so we obtain
\begin{eqnarray*}
I\\
& = \int_{\left(1/t\Fq[[1/t]]\right)^r} \max_{i=1}^r \log|x_i|_v d\nu(x_1)\cdots d\nu(x_r) \\
& = \frac{1}{q^r}\cdot \int_{\left(\Fq[[1/t]]\right)^r} -1+\max_{i=1}^r \log|y_i|_v d\nu(y_1)\cdots d\nu(y_r)\text{ (by the change of variables $x_i=y_i/t$)}\\
& = - \frac{1}{q^r} + \frac{I}{q^r}
\end{eqnarray*} 
So, $I=-\frac{1}{q^r-1}$. 
\end{proof}

So, using \eqref{integral is zero} coupled with Lemma~\ref{integral of h} we obtain that
$$\int_{J_v} g_{\beta,\epsilon}(z) d\mu_v = - \int_{J_v} h_{\beta,\epsilon}(z) d \mu_v =  \frac{\epsilon^r}{q^r-1}  .$$
Using the fact that the measures $ \mu_n$ converge (weakly) to $\mu_v$ (according to \cite{equi}; see also Theorem~\ref{T:equi}) we conclude that for $n$ sufficiently large we have
\begin{equation}
\label{bounding the g}
\left|\int_{J_v} g_{\beta,\epsilon}(z) d\mu_n \right| < 2 \epsilon^r .
\end{equation}
It remains to bound $\left|\int_{J_v} h_{\beta,\epsilon}(z) d\mu_n\right|$. 

Since $h_{\beta,\epsilon}$ is supported on $J_{v,\beta,\epsilon}$ it suffices to analyze the conjugates of $\gamma_n$ which land in $J_{v,\beta,\epsilon}$. For this we let $D_n$ be the smallest integer larger than $[K(\gamma_n):K]^{1/2r}$ and we split $J_{v,\beta,\epsilon}$ into $D_n$ subsets as follows. For each interval $[c, d]\subset [0,1]$ let $J_v([c, d])$ be the subset of $J_{v,\beta,\epsilon}$ containing  all $z\in J_v$ such that $c\epsilon\le |z-\beta|_v\le d\epsilon$. Then 
$$J_{v,\beta,\epsilon}=\bigcup_{i=1}^{D_n} J_v\left(\left[\frac{i-1}{D_n},\frac{i}{D_n}\right]\right).$$
We note that $\mu_v(J_v([c,d]))=(d\epsilon)^r-(c\epsilon)^r$ for each $0\le c<d\le 1$ since $J_{v,\beta,\epsilon}$ is isomorphic to $\left(1/t^m \cdot \Fq[[1/t]]\right)^r$ and $|1/t^m|_v=\epsilon$. So, for large $n$, Corollary~\ref{corollary strong equidistribution} (applied with $\delta :=2/3$, say) yields each such
subset contains at most 
$$2\mu_v(J_v([c,d]))\cdot [K(\gamma_n):K]\le 2\epsilon^r(d^r-c^r)D_n^{2r}.$$ 
conjugates of $\gamma_n$.  The reason for which $\delta =2/3$ works is that in this case
$$\mu_v(J_v([c,d]))\cdot [K(\gamma_n):K]\gg_\epsilon D_n^r\gg D_n^{2r(1-\delta)}\gg [K(\gamma_n):K]^{1-\delta},$$
and thus Corollary~\ref{corollary strong equidistribution} yields that $2\mu_v(J_v([c,d]))\cdot [K(\gamma_n):K]$ is the main term for computing the number of conjugates of $\gamma_n$ landing in $J_v([c,d])$ (note that $\epsilon$ is fixed for this computation).

We analyze the first interval: $J_v([0, 1/D_n])$.  We recall that $d_n$ is the degree of the minimal monic polynomial $Q_n$ such that $\Phi_{Q_n(t)}(\gamma_n)=0$.  Without loss of generality we assume $\gamma_n\in J_v([0, 1/D_n])$ is the conjugate of $\gamma_n$ closest to $\beta$.  By Bosser's theorem (see \cite{Bosser} and also our Fact~\ref{Bosser}), we have
$$\log|\gamma_n-\beta|_v \ge C_0 +C_1d_n\log d_n.$$
On the other hand,  there are at most 
$$2\mu_v(J_v([0,1/D_n]))[K(\gamma_n):K]=\frac{2\epsilon^r[K(\gamma_n):K]}{D_n^r}\le  2\epsilon^r [K(\gamma_n):K]^{1/2}$$ 
conjugates of $\gamma_n$ in $J_v([0, 1/D_n])$. We denote by $I_0$ the set of all $\sigma\in\Gal(\Kbar/K)$ such that $\gamma_n^\sigma\in J_v([0,1/D_n])$. Using also that $\epsilon < 1$, we conclude that
\begin{equation}
\label{estimate on the first arc}
0\ge \int_{J_v([0,1/D_n])} h_{\beta,\epsilon}(z) d\mu_n = \frac{\sum_{\sigma\in I_0}\log\frac{|\beta-\gamma_n^\sigma|_v}{\epsilon}}{[K(\gamma_n):K]}  \ge \frac{4\epsilon^r \cdot (C_0+C_1d_n\log d_n)}{[K(\gamma_n):K]^{1/2}} \ge -\epsilon^r ,
\end{equation}
for large $n$ since $D_n=[K(\gamma_n):K]^{1/2}\gg q^{d_n/2}$ (by \cite{Pink-Rutsche}). 

Finally, consider the remaining subsets of $J_{v,\beta,\epsilon}$. For each $i=2,\dots, D_n$ there are at most $2\mu_v(J_v([(i-1)/D_n,i/D_n]))\cdot [K(\gamma_n):K]$ conjugates of $\gamma_n$ in $J_v([(i-1)/D_n, i/D_n])$ and for each such conjugate $\gamma_n^\sigma$ we have
$$h_{\beta,\epsilon}(\gamma_n^\sigma) =\log\frac{|\beta -\gamma_n^\sigma|_v}{\epsilon} \ge  \log((i-1)/D_n).$$
Using that  $J_{v,\beta,\epsilon}$ is isomorphic to $\left(1/t^m\cdot \Fq[[1/t]]\right)^r$ we conclude
$$
0\ge \int_{J_v\setminus J_v([0,/D_n])} h_{\beta,\epsilon}(z) d\mu_n$$
$$ \ge \sum_{i=1}^{D_n-1}  \log\left(\frac{i}{D_n}\right)\cdot 2\mu_v(J_v[(i-1)/D_n, i/D_n])\\
 > 2 \int_{J_{v,\beta,\epsilon}} \log \left(\frac{|z|_v}{\epsilon}\right)\\
 = -\frac{2\epsilon^r}{q^r-1},
$$
by Lemma~\ref{integral of h}.  
Hence, using also inequality \eqref{estimate on the first arc} we obtain that
\begin{equation}
\label{bounding the h}
\left| \int_{J_v} h_{\beta,\epsilon}(z)d\mu_n \right| < 3\epsilon^r.
\end{equation}
Combining inequalities \eqref{bounding the g} with \eqref{bounding the h} we conclude that
$$\left|\int_{J_v} \log|z-\beta|_v d\mu_n\right| < 5\epsilon^r,$$
for all $n$ sufficiently large, 
and so,
$$\lim_{n\to\infty} \frac{1}{[K(\gamma_n):K]} \cdot \sum_{\sigma\in\Gal(\Kbar/K)} \log|\gamma_n^\sigma-\beta| = 0 = \hhat_{\Phi,v}(\beta),$$
if $\beta\in J_v$. This concludes our proof.
\end{proof}

Propositions~\ref{notin S infinity} and \ref{in S infinity} finish the proof of Theorem~\ref{alternative way of computing the local height}.

\begin{remark}
\label{final remarks}
In the proof of Proposition~\ref{in S infinity} we use in an essential way the hypothesis that $\Phi$ has no complex multiplication. However, we expect that Theorem~\ref{Ih's conjecture result} holds with this hypothesis on $\Phi$, but the proof would be harder.
\end{remark}

\end{document}